\documentclass[12pt,oneside,reqno]{amsart}
\usepackage[all]{xy}
\usepackage{amsfonts,amsmath,oldgerm,amssymb,amscd}
\UseComputerModernTips
\numberwithin{equation}{section}
\usepackage[breaklinks]{hyperref}
\allowdisplaybreaks
\usepackage{color, soul}
\usepackage{xcolor}
\definecolor{revcolor}{RGB}{210, 230, 255} 
\sethlcolor{revcolor}

\usepackage{enumerate}
\usepackage{caption} 
\newtheorem{theorem}{Theorem}[section]
\newtheorem{lemma}[theorem]{Lemma}

\begin{document}

	\title[Diophantine Equations involving associated Pell numbers and repdigits]{Some Diophantine Equations involving associated Pell numbers and repdigits} 

\author[M. Mohapatra]{M. Mohapatra}
\address{Monalisa Mohapatra, Department of Mathematics, National Institute of Technology Rourkela, Odisha-769 008, India}
\email{mmahapatra0212@gmail.com}

\author[P. K. Bhoi]{P. K.  Bhoi}
\address{Pritam Kumar Bhoi, Institute of Mathematics and Applications, Bhubaneswar, Odisha -751 029, India}
\email{pritam.bhoi@gmail.com}

\author[G. K. Panda]{G. K. Panda}
\address{Gopal Krishna Panda, Department of Mathematics, National Institute of Technology Rourkela, Odisha-769 008, India}
\email{gkpanda\_nit@rediffmail.com}


\begin{abstract}
	In this paper, we explore the relationship between repdigits and associated Pell numbers, specifically focusing on two main aspects: expressing repdigits as the difference of two associated Pell numbers, and identifying which associated Pell numbers can be represented as the difference of two repdigits. Additionally, we investigate all associated Pell numbers which are the concatenation of three repdigits. Our proof utilizes Baker's theory on linear forms in logarithms of algebraic numbers, along with the Baker-Davenport reduction technique. The computations were carried out with the help of a simple computer program in {\it Mathematica}.
\end{abstract}

\maketitle
\pagenumbering{arabic}
\pagestyle{headings}

\section{Introduction}
A Diophantine equation is an algebraic or exponential equation with two or more variables intended to be solved in terms of integers only. These types of equations are named after the ancient Greek mathematician Diophantus.\vspace{0.3cm}
\\The associated Pell sequence $(q_n)_{n\geq0}$ is defined by the binary recurrence relation
\begin{equation}
	q_{n+1}=2q_n+q_{n-1},
\end{equation}
with initial conditions $q_0 = 1,\quad q_1 = 1$. The closed form of {the} associated Pell numbers is known as Binet's formula and has the form
	$q_n={(\alpha^{n}+\beta^{n})}/{2}$,
where $(\alpha,\beta)=(1+\sqrt{2},1-\sqrt{2})$ is the pair of roots of the characteristic polynomial $x^2-2x-1$. {This easily implies that the inequality}
\begin{equation}\label{qn}
	\alpha^{n-1} \leq 2q_n <\alpha^{n+1},
\end{equation}
holds for all $n\geq1$.

A palindromic number is one that remains unchanged when its digits are reversed. A special type of palindromic number, known as a repdigit, consists of a single digit repeated multiple times in base 10.  Mathematically, a repdigit can be expressed in the form $d({10^{k}-1) }/{9}$ , where $d\in\{1,2,\ldots,9\}$ and $k\geq1$. Notably, when $k=1$, the result is simply the digit itself, representing a trivial case of a repdigit.

Several authors have explored problems related to repdigits within the context of second-order linear recurrence sequences. All Balancing
and Lucas-balancing numbers which are repdigits have been found in \cite{rp2018}. S. G. Rayaguru and G. K. Panda \cite{rp2021} investigated all balancing and Lucas-balancing numbers which can be expressed as the sums of two repdigits. Additionally, Rayaguru and Bravo \cite{rb2023} identified all Balancing and Lucas-balancing numbers formed by the concatenation of three repdigits. Erduwan et al. \cite{ekl2021} found all Fibonacci and Lucas numbers which are {the} difference of two repdigits. Edjeou and Faye \cite{ef2023}  found all Pell and Pell-Lucas numbers, which are {the} difference of two repdigits. M. G. Duman \cite{d2023} identified all Padovan numbers representable as the difference of two repdigits. More recently, Mohapatra et al. \cite{mbp2024} investigated the existence of repdigits as the difference of two Balancing or Lucas-balancing numbers, and they also enumerated all Balancing and Lucas-balancing numbers that can be expressed as the difference of two repdigits in \cite{mbp2024s}.

This paper seeks to extend previous research by investigating the fascinating realm of repdigits. Specifically, we explore repdigits that can be expressed as the difference between two associated Pell numbers, those formed by concatenating three repdigits, and associated Pell numbers that can be represented as the difference of two repdigits. To facilitate this exploration, we consider the following equations:
\begin{equation}\label{1}
	q_n - q_m = d\biggr(\frac{10^k -1}{9}\biggr).
\end{equation}
{We assume} $k \geq 1$ to avoid trivial solutions.
\begin{equation}\label{2}
	q_n= \overline{\underbrace{d_{1} \ldots d_{1}}_{m_{1} \text { times }}\underbrace{d_{2} \ldots d_{2}}_{m_{2} \text { times }}\underbrace{d_{3} \ldots d_{3}}_{m_{3} \text { times }}},
\end{equation}
where $m_1, m_2, m_3 \geq 1$, $1 \leq d_1 \leq 9$ and $0 \leq d_2, d_3 \leq 9$.
\begin{equation}\label{3}
	q_n = d_1\biggr(\frac{10^k -1}{9}\biggr)- d_2\biggr(\frac{10^l -1}{9}\biggr),
\end{equation}
where $(k, l, n)$ are positive integers with $k > l$, $k\geq2$, and $1 \leq d_1, d_2 \leq 9$.

\section{Auxiliary Results}
To solve the Diophantine equations, we will repeatedly invoke a Baker-type lower bound for a nonzero linear form in the logarithms of algebraic numbers. These lower bounds are instrumental in effectively resolving such equations. We shall commence by revisiting essential definitions and key results from the realm of algebraic number theory.

Let $\lambda$ be an algebraic number with minimal primitive polynomial 
\[f(X) = a_0(X-\lambda^{(1)})\cdot\cdot\cdot(X-\lambda^{(k)}) \in \mathbb{Z}[X],\]
where $a_0 > 0$ is the leading coefficient and $\lambda^{(i)}$'s are conjugates of $\lambda$. Then the $absolute$ $logarithmic$ $height$ of $\lambda$ is given by
\begin{equation*}
	h(\lambda) = \dfrac{1}{k}\biggr(\log a_0 + \sum\limits_{j=1}^{k}\max\{0,\log|\lambda^{(j)}|\}\biggr).
\end{equation*}
If $\lambda = a/b$ is a rational number with $\gcd(a,b)=1$ and $b>1$, then  $h(\lambda)=\log$(max$\{|a|,b\})$.  
Here are some properties of the $absolute$ $logarithmic$ $height$ whose proofs can be found in \cite{bms2006}. Let $\gamma$ and $\eta$ be two algebraic numbers, then 
	\begin{enumerate}
	    \item[(i)] $\hspace{0.2cm} h(\gamma\pm\eta) \leq h(\gamma)+${$h(\eta)$}$+\log 2,$
		\item[(ii)] $\hspace{0.2cm} h(\gamma\eta^{\pm 1})\leq h(\gamma)+h(\eta),$
		\item[(iii)]$\hspace{0.2cm} h(\gamma ^k)=|k|h(\gamma).$
	\end{enumerate}

Building on the previous notations, we present a theorem that refines a result by Matveev \cite{m2000}, as extended by Bugeaud et al. \cite{bms2006}. This theorem offers a precise upper bound for our variables in equations \eqref{1}, \eqref{2}, and \eqref{3}.

\begin{theorem}\label{thm1}\cite{m2000}. Let $\gamma_1,\ldots,\gamma_l$ be positive real numbers in an algebraic number field $\mathbb{L}$ of degree $d_{\mathbb{L}}$ and $b_1, \ldots, b_l$ be nonzero integers. If $\Gamma = \prod\limits_{i=1}^{l} \gamma_{i}^{b_i} -1$ is not zero, then 
	\begin{equation*}
		\log|\Gamma| > -1.4\cdot 30^{l+3}\cdot l^{4.5} \cdot d_{\mathbb{L}}^2 (1+\log d_{\mathbb{L}}^2)(1+\log(D))A_1 A_2 \cdot\cdot\cdot A_l,
	\end{equation*}
	where $D\geq$ $\max\{|b_1|,\ldots,|b_l|\}$ and  $A_1, \cdots, A_l$ are positive integers such that $A_j \geq h'(\gamma_j)$ = $\max\{d_{\mathbb{L}}h(\gamma_j),|\log \gamma_j|, 0.16\},$ for $j= 1,\ldots,l$.
\end{theorem}

The subsequent result, recognized as the Baker-Davenport theorem and credited to Dujella and Peth\H{o} \cite{dp1998} is another tool in our proofs. It will be used to reduce the upper bounds on our variables.
\begin{lemma}\label{lem1}\cite{dp1998}. Let $M$ be a positive integer and $p/q$ denote a convergent of the continued fraction of the real number $\tau$ such that $q > 6M$. Consider the real numbers $A, B, \mu$ with $A > 0$ and $B > 1$. Let $\epsilon := \left\lVert \mu q\right\rVert - M\left\lVert \tau q\right\rVert$, where $\left\lVert .\right\rVert$ denotes the distance from the nearest integer. If $\epsilon > 0$, then there exists no solution to the inequality
	\begin{equation}
		0 <|u\tau - v +\mu|<AB^{-w},
	\end{equation}
	in positive integers $u,v,w$ with $u\leq M$ and $w\geq {\log(Aq/\epsilon)}/{\log B}$.
\end{lemma}

We conclude this section by recalling the following lemma that we need in the sequel:
\begin{lemma}\label{lem2}\cite{sl2014}. Let $r\geq 1$ and $H > 0$ be such that $H > (4r^2)^r$ and $H>L/(\log L)^r$. Then $L < 2^rH(\log H)^r$.
\end{lemma}
\begin{lemma}\label{lem3}\cite{wb1989}. 
	Let $a,x\in \mathbb R$. If $0<a<1$ and $|x|<a$, then 
	\begin{equation*}
		|\log(1+x)|<\dfrac{-|\log(1-a)|}{a}.|x| \text{ and } |x|<\dfrac{a}{1-e^{-a}}.|e^x - 1|.
	\end{equation*} 
\end{lemma}
\begin{lemma}\label{lem4}\cite{rps2021}
	The only associated Pell numbers which are repdigits are 1, 3, 7, and 99.
\end{lemma}
\begin{lemma}\label{lem5}\cite{rps2021}
	The only associated Pell numbers which are the concatenation of two repdigits are 17, 41, and 577.  
\end{lemma}

\section{Repdigits as the difference of two associated Pell numbers}

\begin{theorem}\label{thm3}All solutions of the Diophantine Equation \eqref{1} in positive integers  with $n>m$, $k\geq1$, $1\leq d\leq9$, are 
$(n,m,d,k)$ $\in$ \{(2,0,2,1), (2,1,2,1), (3,2,4,1), (3,0,6,1), (3,1,6,1), (7,4,4,3)\}.

\end{theorem}

\begin{proof}  Using $Mathematica$, we get the repdigits as the difference of two associated Pell numbers for $n\leq100$ as listed in Theorem \ref{thm3}. So, assume that $n>100$. By \eqref{qn} and \eqref{1}, the inequality \begin{equation*}
		10^{k-1}<{d(10^k-1)}/{9}=q_n-q_m\leq2q_{n}<\alpha^{n+1}<10^{n+1},
		\end{equation*}
    implies that $ k<n+2.$
Using Binet's formula, \eqref{1} can be written as
	\begin{equation}\label{9}
		\frac{\alpha^{n}+\beta^{n}}{2} - \frac{\alpha^{m}+\beta^{m}}{2} = d\biggr(\frac{10^{k}-1 }{9}\biggr).
	\end{equation}
	We will now rearrange equation \eqref{9} into two distinct cases, as outlined below.
 \subsection{Case 1} First rearrangement of \eqref{9} is $$\frac{\alpha^{n}}{2} - \frac{d\cdot10^k}{9} = \frac{\alpha^{m}}{2} + \frac{\beta^{m}}{2} - \frac{\beta^{n}}{2} - \frac{d}{9}.$$ By applying the absolute value to both sides, we arrive at
	\begin{equation*}
		\biggr|\frac{\alpha^{n}}{2} - \frac{d\cdot 10^k}{9}\biggr| \leq \biggr|\frac{\alpha^{m}}{2}\biggr| + 4 \leq \frac{9\alpha^{m}}{2}.
	\end{equation*} 
	After applying $\frac{\alpha^{n}}{2}$ as a divisor to both sides of the inequality, we find
	\begin{equation}\label{10}
		{\biggr|1- \alpha^{-n}10^{k}\biggr(\frac{2d}{9}\biggr)\biggr|} <\frac{9}{\alpha^{n-m}}.
	\end{equation}
	{Let }$\Gamma_1$ = $1-\alpha^{-n}10^{k}({2d}/{9})$. It is enough to check that $\Gamma_1 \neq 0$. On the contrary, suppose $\Gamma_1=0$, then $\alpha^{n}=10^{k} ({2d}/{9}) \in \mathbb{Q}$, which contradicts the fact that $\alpha^{n}$ is irrational for any $n > 0$. Therefore, $\Gamma_1 \neq 0$.
	To implement Theorem \ref{thm1}, define 
	\begin{equation*}
		\lambda_1 = \alpha, \lambda_2 = 10, \lambda_3 = {2d}/{9}, b_1 = -n, b_2 = k, b_3 = 1, l = 3,
	\end{equation*}  where $\lambda_1$, $\lambda_2$, $\lambda_3$ $\in \mathbb{Q}[\sqrt{2}]$ and $b_1, b_2, b_3 \in \mathbb{Z}$. It can be observed that $\mathbb{Q}(\lambda_1, \lambda_2, \lambda_3) = \mathbb{Q}(\alpha)$, so $d_\mathbb{L} = 2$.
	Since $k < n+2$, we have $D=$max$\{n, k, 1\} = n+2$. The absolute logarithmic heights of $\lambda_1$, $\lambda_2$ and $\lambda_3$ are computed as $h(\lambda_1)=(\log\alpha)/2$, $h(\lambda_2)= \log 10$ and $h(\lambda_3)\leq \log(2\cdot 8)=\log 16$ .
	 Thus, we can express the following:
	\begin{equation*}
	  \begin{split}  
	 & \max\{2h(\lambda_1),|\log \lambda_1|, 0.16 \} = \log \alpha = A_1,\\
  & \max\{2h(\lambda_2),|\log \lambda_2|, 0.16 \} = 2\log 10 = A_2,\\ 
  & \max\{2h(\lambda_3),|\log \lambda_3|, 0.16 \} \leq \log 16 = A_3. \end{split} \end{equation*}
	Now, utilizing Theorem \ref{thm1}, we can determine a lower bound for $\log|\Gamma_1|$ as \begin{equation*}
		\log|\Gamma_1| > -1.4\cdot30^6\cdot3^{4.5}\cdot2^2 (1+\log 2)(1+\log (n+2))(\log\alpha)(2\log 10)(\log 16).\end{equation*}
	By comparing the inequality presented above with \eqref{10}, we find that
	\begin{equation}\label{11}
		(n-m)\log\alpha < \log 9 + 1.1 \cdot 10^{13} (1+\log (n+2)) < 1.2 \cdot 10^{13} (1+\log (n+2)).
	\end{equation}
	\subsection{Case 2} We perform the second rearrangement of \eqref{9} as
	\begin{center}
		$\dfrac{\alpha^{n}}{2} - \dfrac{\alpha^{m}}{2} - \dfrac{d\cdot 10^k}{9} =  \dfrac{\beta^{m}}{2} - \dfrac{\beta^{n}}{2} - \dfrac{d}{9}$.
		\end{center}
	By taking the absolute values on both sides of the inequality, we arrive at
	\begin{equation*}
		\biggr|\frac{\alpha^{n}}{2} - \frac{\alpha^{m}}{2} - \frac{d\cdot 10^k}{9}\biggr| \leq 4.
	\end{equation*} 
	When we divide both sides of the inequality by $\alpha^{n}(1-\alpha^{m-n})/2$, it yields
	\begin{equation}\label{12}
		{\biggr|1- \alpha^{-n}10^{k}\biggr(\frac{2d}{9(1-\alpha^{m-n})}\biggr)\biggr|} <\frac{5}{\alpha^{n}}.
	\end{equation}
	Define $\Gamma_2$ = $1- \alpha^{-n}10^{k}({2d}/{9(1-\alpha^{m-n})})$. In the same way, it can be shown that $\Gamma_2 \neq 0$. Here we have $h(\lambda_1)=h(\alpha)=(\log\alpha)/2$ and $h(\lambda_2)=h(10)=\log 10$. Let $\lambda_3$ = ${2d}/({9(1-\alpha^{m-n})})$. Then, 
	\begin{equation*}
		\begin{split}
			h(\lambda_3) & \leq h(2d) + h(9(1-\alpha^{m-n}))\\
			& \leq 2\log 9+2\log 2+(n-m)\frac{\log\alpha}{2}\\
			& < 5.8+ (n-m) \frac{\log\alpha}{2}.
		\end{split}
	\end{equation*}
	Accordingly, we define $A_3 =11.6+ (n-m)\log\alpha$.
	 By virtue of Theorem \ref{thm1},
    \begin{equation}\label{11.1}
        \begin{split}
			\log|\Gamma_2| >  -& 1.4\cdot30^6\cdot3^{4.5}\cdot2^2 (1+\log 2)(1+\log (n+2))(\log\alpha)(2\log 10)\\
			& \times(11.6+ (n-m)\log\alpha).
		\end{split}
    \end{equation}
    Putting the value of $n-m$ from \eqref{11} we derive 
	\begin{equation*}
		\begin{split}
			\log|\Gamma_2| >  -& 1.4\cdot30^6\cdot3^{4.5}\cdot2^2 (1+\log 2)(1+\log (n+2))(\log\alpha)(2\log 10)\\
			& \times(1.3\cdot 10^{13})(1+\log (n+2)).
		\end{split}
	\end{equation*}
    A comparison of the above inequality with \eqref{12} yields
	\begin{equation*}
		n\log\alpha < \log 5 + 5.2 \cdot 10^{25} (1+\log (n+2))^2 < 8 \cdot 10^{25} (\log n)^2,
	\end{equation*}
    where we have used the fact that $(1+\log (n+2))^2<1.5 (\log n)^2$ for all $n>100.$
	With the notations of Lemma \ref{lem2}, we take $r=2, L=2, H= 8\cdot10^{25}/\log \alpha$ to get 
	\begin{equation*}
		n < 2^2\biggr(\dfrac{8.5\cdot10^{25}}{\log\alpha}\biggr) \biggr(\log\biggr(\dfrac{8.5\cdot10^{25}}{\log \alpha}\biggr)\biggr)^2 < 1.3\cdot10^{30}.
	\end{equation*}
	\subsection{Reducing the upper bound of $n$}
	We now need to apply the Baker-Davenport reduction method, as developed by Dujella and Peth\H{o} \cite{dp1998}, to refine the bound. Let 
	\begin{equation*}
		\Lambda_1 = -n\log\alpha + k\log 10 + \log({2d}/{9}) .
	\end{equation*}
	We can rewrite the inequality \eqref{10} as
	\begin{equation*}|e^{\Lambda_1} - 1| < {9}/{\alpha^{n-m}}.
		\end{equation*}
	Observe that $\Lambda_1 \neq 0$ as $e^{\Lambda_1} - 1 = \Gamma_1 \neq 0$.
	Under the assumption that $n-m\geq 4$, the right-hand side of the above inequality is at most ${9}/{(1+\sqrt{2})^4}<{1}/{2}$. From the inequality $|e^{z} - 1| < y$ with real values of $z$ and $y$, we deduce that $z< 2y$. This leads us to the result $|\Lambda_1| < {18}/{\alpha^{n-m}}$, which indicates that
	\begin{center}$\biggr|-n\log\alpha + k\log 10 + \log({2d}/{9}) \biggr| < {18}/{\alpha^{n-m}}$.
	\end{center}
	Upon dividing both sides of the inequality by $\log \alpha,$ we find\begin{equation}\label{13}
		\biggr|k\biggr(\frac{\log 10}{\log\alpha}\biggr) - n + \frac{\log({2d}/{9})}{\log\alpha}\biggr| <\frac{21}{\alpha^{n-m}}.
	\end{equation}
	In order to implement Lemma \ref{lem1}, let us define \begin{center}
		$u=k, \tau=\dfrac{\log 10}{\log\alpha}, v = n, \mu= \biggr(\dfrac{\log({2d}/{9})}{\log\alpha}\biggr), A= 21, B= \alpha, w = n-m$. 
	\end{center}
	We can set $M= 1.3\cdot10^{30}$ as an upper bound on $u$. The denominator of 68-th convergent of $\tau$, denoted as $q_{68}= 27232938992914655197439992935676$, exceeds $6M$. Considering the fact that $1\leq d\leq 9$, a quick computation with $Mathematica$ gives {us} the least possible positive value of $\epsilon := \left\lVert \mu q_{68}\right\rVert - M\left\lVert \tau q_{68}\right\rVert = 0.131525$ for $d=8$. Applying Lemma \ref{lem1} to the inequality \eqref{13}, we obtain $n-m\leq 87$. Consequently, substituting this upper bound for $n-m$
into \eqref{11.1}, we get $n<2.4\cdot 10^{15}$. Now, let \begin{equation*}
		\Lambda_2 = -n\log\alpha + k\log 10 + \log\biggr(\dfrac{2d}{9(1-\alpha^{m-n})}\biggr). 
	\end{equation*}
	The inequality \eqref{12} can be expressed as \begin{equation*}|e^{\Lambda_2} - 1| < \dfrac{5}{\alpha^{n}}.
	\end{equation*}
	Observe that $\Lambda_2 \neq 0$ as $e^{\Lambda_2} - 1 = \Gamma_2 \neq 0$.
	Assuming $n\geq 2$, {the right-hand side of the above inequality is at most $1/2$.} The inequality $|e^{z} - 1| < y$ for real values of $z$ and $y$ leads to the conclusion that
 $z < 2y$. Therefore, we can assert $|\Lambda_2| < \frac{10}{\alpha^{n}}$, which implies that \begin{equation*}
\biggr|-n\log\alpha + k\log 10 + \log\biggr(\dfrac{2d}{9(1-\alpha^{m-n})}\biggr) \biggr| < \dfrac{10}{\alpha^{n}}.\end{equation*}Dividing both sides of the above inequality by $\log \alpha$ yields
	\begin{equation}\label{14}
		\biggr|k\biggr(\frac{\log 10}{\log\alpha}\biggr) - n + \dfrac{\log (\frac{2d}{9(1-\alpha^{m-n})})}{\log\alpha}\biggr| <\frac{12}{\alpha^{n}}.
	\end{equation}
	Let \begin{equation*}u=k, \tau=\dfrac{\log 10}{\log\alpha}, v = n, \mu= \biggr(\dfrac{\log(\frac{2d}{9(1-\alpha^{m-n})})}{\log\alpha}\biggr), A= 12, B= \alpha, w = n.
	\end{equation*}
	 Choose $M = 2.4\cdot 10^{15}$. We find $q_{41}= 3091088636788945 $, the denominator of 41-th convergent of $\tau$ exceeds $6M$. For $1\leq d\leq 9$, and $0\leq n-m \leq 87,$ a quick computation with $Mathematica$ gives us the minimum positive $ \epsilon := \left\lVert \mu q_{41}\right\rVert - M\left\lVert \tau q_{41}\right\rVert = 0.0002045763 $ for $d=6$ and $n-m=30$. Now, applying Lemma \ref{lem1} we find that the inequality \eqref{14} has no solutions for $n\geq55.53$. So $n\leq55.$ This conclusion directly contradicts our initial assumption that $n > 100$.
\end{proof}

\section{associated Pell numbers which are the concatenation of three repdigits} 
\begin{theorem}\label{thm4}
	The only associated Pell numbers which are the concatenation of three repdigits are 239, 3363, and 8119.
\end{theorem}
\begin{proof}
	
	Assuming that \eqref{2} holds, we examined the first 100 associated Pell numbers and found that the solutions to the Diophantine equation \eqref{2} are $q_{n} \in\{239, 3363, 8119\}$ for $d_{1}, d_{2} \in\{0,1, \ldots, 9\}$ with $d_{1}>0$. From this point forward, we assume that $n > 100$. The scenarios $d_{1}=d_{2} \neq d_{3}$ and $d_{1} \neq d_{2}=d_{3}$ in \eqref{2} are ruled out, as the only associated Pell numbers, which are the concatenation of two repdigits are 17, 41, 577 by Lemma \ref{lem5}. Furthermore, the case $d_{1}=d_{2}=d_{3}$ in \eqref{2} is also impossible since the largest repdigit in the associated Pell sequence is 99 as stated in Lemma \ref{lem4}.\vspace{0.5cm}
 \\Let us define
	\begin{equation*}
		q_n= \overline{\underbrace{d_{1} \ldots d_{1}}_{m_{1} \text { times }}\underbrace{d_{2} \ldots d_{2}}_{m_{2} \text { times }}\underbrace{d_{3} \ldots d_{3}}_{m_{3} \text { times }}} \hspace{0.1cm}.\end{equation*}
	Then we can express $q_n$ as:
	$$
	q_{n}=\underbrace{d_{1} \ldots d_{1}}_{m_{1} \text { times }} \cdot 10^{m_{2}+m_{3}}+\underbrace{d_{2} \ldots d_{2}}_{m_{2} \text { times }} \cdot 10^{m_{3}}+\underbrace{d_{3} \ldots d_{3}}_{m_{3} \text { times }} \hspace{0.1cm},
	$$
	which leads us to
	\begin{equation}\label{15}
		q_{n}=\frac{d_{1}\left(10^{m_{1}}-1\right)}{9} 10^{m_{2}+m_{3}}+\frac{d_{2}\left(10^{m_{2}}-1\right)}{9} 10^{m_{3}}+\frac{d_{3}\left(10^{m_{3}}-1\right)}{9}\hspace{0.1cm}.
	\end{equation}
	Alternatively, we can express it as:
	\begin{equation}\label{16}
		q_{n}=\frac{1}{9}\left(d_{1} 10^{m_{1}+m_{2}+m_{3}}-\left(d_{1}-d_{2}\right) 10^{m_{2}+m_{3}}-\left(d_{2}-d_{3}\right) 10^{m_{3}}-d_{3}\right) . 
	\end{equation}
	Combining the {right-hand }side of inequality \eqref{qn} with \eqref{15}, we arrive at the following relationship:
	$$
	10^{m_{1}+m_{2}+m_{3}-1}<q_{n}<\alpha^{n+1}<10^{n+1}.
	$$
	From this, we can conclude that: $m_{1}+m_{2}+m_{3}<n+2$. 
 We will now rearrange equation \eqref{16} in three distinct cases by using the Binet's formula of associated Pell numbers.

 \subsection{Case 1} The first rearrangement of \eqref{16} is given by
	\begin{equation}\label{17}
		\frac{9 \alpha^{n}}{2}-d_{1} 10^{m_{1}+m_{2}+m_{3}}=-\biggr(\frac{9 \beta^{n}}{2}+\left(d_{1}-d_{2}\right) 10^{m_{2}+m_{3}}+\left(d_{2}-d_{3}\right) 10^{m_{3}}-d_{3}\biggr) .
	\end{equation}
	Taking the absolute values of both sides of \eqref{17} yields
	$$
	\begin{aligned}
		\left|\frac{9 \alpha^{n}}{2}-d_{1} 10^{m_{1}+m_{2}+m_{3}}\right| & \leq \frac{9 \beta^{n}}{2}+\left(d_{1}-d_{2}\right) 10^{m_{2}+m_{3}}+\left(d_{2}-d_{3}\right) 10^{m_{3}}+d_{3} \\
		& \leq \frac{9 \alpha^{-n}}{2}+9 \cdot 10^{m_{2}+m_{3}}+9 \cdot 10^{m_{3}}+9 \\
		& \leq \frac{9 \alpha^{-n}}{2}+\left(9 \cdot 10^{m_{2}+m_{3}}+9 \cdot 10^{m_{3}}+0.9 \cdot 10^{m_{3}}\right) \\
		& =\frac{9 \alpha^{-n}}{2}+10^{m_{3}}\left(9 \cdot 10^{m_{2}}+9.9\right) \\
		& \leq \frac{9 \alpha^{-n}}{2}+10^{m_{3}}\left(9 \cdot 10^{m_{2}}+0.99 \cdot 10^{m_{2}}\right) \\
		& <\frac{0.09 \cdot \alpha^{-n} \cdot 10^{m_{2}+m_{3}}}{2}+9.99 \cdot 10^{m_{2}+m_{3}} \\
		& <9.991 \cdot 10^{m_{2}+m_{3}},
	\end{aligned}
	$$
	where $n >100$. Therefore,
	\begin{equation}\label{18}
		\left|\frac{9 \alpha^{n}}{2}-d_{1} 10^{m_{1}+m_{2}+m_{3}}\right|<9.991 \cdot 10^{m_{2}+m_{3}} .
	\end{equation}
	Dividing both sides of \eqref{18} by $d_{1} 10^{m_{1}+m_{2}+m_{3}}$  gives us
	\begin{equation}\label{19}
		\left|\frac{9}{2 d_{1}} \alpha^{n} 10^{-m_{1}-m_{2}-m_{3}}-1\right|<\frac{9.991}{10^{m_{1}}}.
	\end{equation}
	Now, let us apply Theorem \ref{thm1} with the parameters $\gamma_{1}:=9 /\left(2 d_{1}\right), \gamma_{2}:=\alpha, \gamma_{3}:=10$ and $b_{1}:=1, b_{2}:=n$, $b_{3}:=-m_{1}-m_{2}-m_{3}$. It is worth noting that $\gamma_{1}, \gamma_{2}$, and $\gamma_{3}$ are positive real numbers and belong to the field $\mathbb K=\mathbb Q[\sqrt{2}]$, which has degree $d_{\mathbb L} = 2$. We define:
	$$
	\Gamma_3:=(({9}/{2 d_{1}}) \alpha^{n} 10^{-m_{1}-m_{2}-m_{3}})-1.
	$$
	Setting $\Gamma_3=0$, we arrive at
	$$
	\alpha^{n}={(2 d_{1}}/{9}) \cdot 10^{m_{1}+m_{2}+m_{3}}.
	$$
	However, this leads to a contradiction as $\alpha^{n}$ is irrational for $n \geq 1$, which implies $\Gamma_3$ must indeed be nonzero. Furthermore, utilizing the properties of absolute logarithmic height, we can obtain 
	$$
	h\left(\gamma_{1}\right)<2.9,
	h\left(\gamma_{2}\right)={(\log \alpha)}{/2}<0.45,\text{ and } h\left(\gamma_{3}\right)=\log 10<2.31 .
	$$
 Now, we can assign values $A_{1}:=5.8, A_{2}:=0.9$, and $A_{3}:=4.62$.
	 Since, $m_{1}+m_{2}+m_{3}<n+2$ and $D \geq \max \left\{|1|,|n|,\left|-m_{1}-m_{2}-m_{3}\right|\right\}$, we can conveniently set $D:=n+2$. Let us define:
	$
	C:=1.4 \cdot 30^{6} \cdot 3^{4.5} \cdot 2^{2} \cdot(1+\log 2) \cdot 0.9 \cdot 4.62.
	$
	Analyzing  inequality \eqref{19} in conjunction with Theorem \ref{thm1}, we obtain
	$$
	9.991 \cdot 10^{-m_{1}}>\left|\Gamma_3\right|>\exp (-C \cdot(1+\log (n+2)) \cdot 5.8).
	$$
	A straightforward computation yields the inequality
	\begin{equation}\label{20}
		m_{1} \log 10< 2.3\cdot 10^{13} \cdot(1+\log (n+2))+\log 9.991 .
	\end{equation}
\subsection{Case 2}
 
Proceeding with the second rearrangement of equation \eqref{16} as
	\begin{equation}\label{21}
		\frac{9 \alpha^{n}}{2}-\left(d_{1} 10^{m_{1}}-\left(d_{1}-d_{2}\right)\right) 10^{m_{2}+m_{3}}=-\biggr(\frac{9 \beta^{n}}{2}+\left(d_{2}-d_{3}\right) 10^{m_{3}}+d_{3}\biggr) 
	\end{equation}
	and taking absolute values of both sides of \eqref{21}, we get
 $$ \biggr|\frac{9 \alpha^{n}}{2}-\left(d_{1} 10^{m_{1}}-\left(d_{1}-d_{2}\right)\right) 10^{m_{2}+m_{3}}\biggr|=\biggr|-\biggr(\frac{9 \beta^{n}}{2}+\left(d_{2}-d_{3}\right) 10^{m_{3}}+d_{3}\biggr)\biggr| .$$
	This equality leads to a series of inequalities based on the properties of absolute values as$$
	\begin{aligned}
		\left|({9 \alpha^{n}}/{2})-\left(d_{1} 10^{m_{1}}-\left(d_{1}-d_{2}\right)\right) 10^{m_{2}+m_{3}}\right| & \leq \frac{9 \beta^{n}}{2}+\left(d_{2}-d_{3}\right) 10^{m_{3}}+d_{3} \\
		& \leq \frac{9 \alpha^{-n}}{2}+9 \cdot 10^{m_{3}}+9 \\
		& \leq \frac{9 \alpha^{-n}}{2}+\left(9 \cdot 10^{m_{3}}+0.9 \cdot 10^{m_{3}}\right) \\
		& <\frac{(0.9) \alpha^{-n} 10^{m_{3}}}{2}+9.9 \cdot 10^{m_{3}} \\
		& <9.91 \cdot 10^{m_{3}}
	\end{aligned}
	$$
	i.e.
	\begin{equation}\label{22}
		\left|({9 \alpha^{n}}/{2})-\left(d_{1} 10^{m_{1}}-\left(d_{1}-d_{2}\right)\right) 10^{m_{2}+m_{3}}\right|<9.91 \cdot 10^{m_{3}} .
	\end{equation}
	Dividing both sides of \eqref{22} by $\left(d_{1} 10^{m_{1}}-\left(d_{1}-d_{2}\right)\right) 10^{m_{2}+m_{3}}$, we arrive at
	\begin{equation}\label{23}
		\left|1-\left(\frac{9}{2\left(d_{1} 10^{m_{1}}-\left(d_{1}-d_{2}\right)\right)}\right) \alpha^{n} 10^{-m_{2}-m_{3}}\right|<\frac{1.11}{10^{m_{2}}} .
	\end{equation}
	Let us introduce the following parameters:
	$$
	\biggr(\gamma_{1}, \gamma_{2}, \gamma_{3}\biggr)=\biggr(\frac{9}{2\left(d_{1} 10^{m_{1}}-\left(d_{1}-d_{2}\right)\right)}, \alpha, 10\biggr)
	$$
	and $$(b_{1},b_{2},b_{3})=(1, n, -m_{2}-m_{3}).$$ With these definitions established, we can proceed to apply Theorem \ref{thm1} . Here, $d_{\mathbb L}=2$ as the values $\gamma_{1}, \gamma_{2}$, and $\gamma_{3}$ are positive real numbers and elements of the field $\mathbb{K}=\mathbb{Q}[\sqrt{2}]$. We define
	$$
	\Gamma_4:=1-\left(\frac{9}{2\left(d_{1} 10^{m_{1}}-\left(d_{1}-d_{2}\right)\right)}\right) \alpha^{n} 10^{-m_{2}-m_{3}}.
	$$
	Using the same arguments applied earlier for $\Gamma_3$, we conclude that $\Gamma_4 \neq 0$. By employing the properties of the absolute logarithmic height, we obtain
	$$
	\begin{aligned}
		h\left(\gamma_{1}\right) & =h\left(\frac{9}{2\left(d_{1} 10^{m_{1}}-\left(d_{1}-d_{2}\right)\right)}\right)<2.9+ m_1 \log 10,\\
		h\left(\gamma_{2}\right) & =h(\alpha)=\frac{\log \alpha}{2}, \\
		h\left(\gamma_{3}\right) & =h(10)=\log 10<2.31 .
	\end{aligned}
	$$
	So, we can assign $A_{1}:=5.8+2 m_{1} \log 10, A_{2}:=0.9$, and $A_{3}:=4.62$. As $m_{2}+m_{3}\leq n$ and $D \geq \max \left\{|1|,|-n|,\left|-m_{2}-m_{3}\right|\right\}$, we can take $D:=n$. Considering the inequality \eqref{23} and applying Theorem \ref{thm1}, we obtain
	$$
	1.11 \cdot 10^{-m_{2}}>\left|\Gamma_4\right|>\exp \left(-C \cdot(1+\log n)\left(5.8+2 m_{1} \log 10\right)\right) .
	$$
	From this, we derive
	\begin{equation}\label{24}
		m_{2} \log 10<4 \cdot 10^{12} \cdot(1+\log n)\left(5.8+2 m_{1} \log 10\right)+\log 1.11 .
	\end{equation}
\subsection{Case 3}	
	
	Carrying out the third rearrangement of equation \eqref{16}, we have
	\begin{equation}\label{25}
		\frac{9 \alpha^{n}}{2}-\left(d_{1} 10^{m_{1}+m_{2}}-\left(d_{1}-d_{2}\right) 10^{m_{2}}-\left(d_{2}-d_{3}\right)\right) 10^{m_{3}}=-\biggr(\frac{9 \beta^{n}}{2}+d_{3}\biggr) .
	\end{equation}
	By taking the absolute values of both sides of equation \eqref{25}, we arrive at
	$$
	\left|\frac{9 \alpha^{n}}{2}-\left(d_{1} 10^{m_{1}+m_{2}}-\left(d_{1}-d_{2}\right) 10^{m_{2}}-\left(d_{2}-d_{3}\right)\right) 10^{m_{3}}\right| \leq \frac{9 \beta^{n}}{2}+d_{3}=\frac{9 \alpha^{-n}}{2}+9<9.1.
	$$
	This leads us to the pivotal inequality
	\begin{equation}\label{26}
		\left|({9 \alpha^{n}}/{2})-\left(d_{1} 10^{m_{1}+m_{2}}-\left(d_{1}-d_{2}\right) 10^{m_{2}}-\left(d_{2}-d_{3}\right)\right) 10^{m_{3}}\right|<9.1.
	\end{equation}
	Upon dividing both sides of \eqref{26} by $9 \alpha^{n} /2$, we obtain
	\begin{equation}\label{27}
		\left|1-\left(\frac{2\left(d_{1} 10^{m_{1}+m_{2}}-\left(d_{1}-d_{2}\right) 10^{m_{2}}-\left(d_{2}-d_{3}\right)\right)}{9}\right) \alpha^{-n} 10^{m_{3}}\right| \leq 2.1\cdot \alpha^{-n}. 
	\end{equation}
	Let us introduce the following parameters:
	$$
	\gamma_{1}:=\left(\frac{2\left(d_{1} 10^{m_{1}+m_{2}}-\left(d_{1}-d_{2}\right) 10^{m_{2}}-\left(d_{2}-d_{3}\right)\right)}{9}\right), \quad \gamma_{2}:=\alpha, \quad \gamma_{3}:=10 .
	$$
	Additionally, we set $b_{1}:=1, b_{2}:=-n, b_{3}:=m_{3}$. This configuration allows us to apply Theorem \ref{thm1}. The parameters $\gamma_{1}, \gamma_{2}$, and $\gamma_{3}$ are all positive real numbers that lie in the field $\mathbb{K}=\mathbb{Q}[\sqrt{2}]$ implying that $d_{\mathbb L}=2$. Let
	$$
	\Gamma_5:=1-\left(\frac{2\left(d_{1} 10^{m_{1}+m_{2}}-\left(d_{1}-d_{2}\right) 10^{m_{2}}-\left(d_{2}-d_{3}\right)\right)}{9}\right) \alpha^{-n} 10^{m_{3}}.
	$$
	We can confirm that $\Gamma_5 \neq 0$, just as for $\Gamma_3$. By leveraging the properties of the absolute logarithmic height, we obtain
	$$
	\begin{aligned}
		h\left(\gamma_{1}\right) & =h\left(\frac{2\left(d_{1} 10^{m_{1}+m_{2}}-\left(d_{1}-d_{2}\right) 10^{m_{2}}-\left(d_{2}-d_{3}\right)\right)}{9}\right) \\
		&\leq\log 20+\left(m_{1}+m_{2}\right) \log 10\\
        &<6+\left(m_{1}+m_{2}\right) \log 10\\
		h\left(\gamma_{2}\right) & =h(\alpha)=\frac{\log \alpha}{2} \\
		h\left(\gamma_{3}\right) & =h(10)=\log 10<2.31 .
	\end{aligned}
	$$
	So, we assign $A_{1}:=12+2 (m_{1}+m_2) \log 10+2m_{2} \log 10, A_{2}:=0.9$, and $A_{3}:=4.62$. As $m_{3}<n-1$ and $D \geq \max \left\{|1|,|-n|,\left|m_{3}\right|\right\}$, we can take $D:=n$. Consequently, by taking the inequality \eqref{27} into account and invoking Theorem \ref{thm1}, we derive
	$$
	2.1 \cdot \alpha^{-n}>\left|\Gamma_5\right|>\exp \left(-C \cdot(1+\log n)\left(12+2 m_{1} \log 10+4 m_{2} \log 10\right)\right)
	$$
	or
	\begin{equation}\label{28}
		n \log \alpha-\log (2.1)<4 \cdot 10^{12} \cdot(1+\log n)\left(12+2 m_{1} \log 10+4 m_{2} \log 10\right). 
	\end{equation}
	Leveraging the inequalities \eqref{20}, \eqref{24}, and \eqref{28}, a computation with $Mathematica$ yields the result that $n<3.8 \cdot 10^{41}$. 
\subsection{Reducing the upper bound on $n$}	
	Now, let us try to tighten the upper bound on $n$ by applying Lemma \ref{lem2}. We define
	$$
	z_{1}:=\log \left(\Gamma_3+1\right)=\left(m_{1}+m_{2}+m_{3}\right) \log 10-n \log \alpha-\log (9 /(2 d_{1}))).
	$$
	From \eqref{19}, we can deduce that
	$$
	\left|\Gamma_3\right|=\left|e^{-z_{1}}-1\right|<\frac{9.991}{10^{m_{1}}}<0.9995 \text { for } m_{1} \geq 1.
	$$
   By selecting $a:=0.9995$, we derive the inequality
	$$
	\left|z_{1}\right|=\left|\log \left(\Gamma_3+1\right)\right|<\frac{\log 2000}{0.9995} \cdot \frac{9.991}{10^{m_{1}}}<\frac{75.98}{10^{m_{1}}}
	$$
	according to Lemma \ref{lem3}. Thus, it follows that
	$$
	0<\left|\left(m_{1}+m_{2}+m_{3}\right) \log 10-n \log \alpha-\log \left(\frac{9}{2 d_{1}}\right)\right|<\frac{75.98}{10^{m_{1}}} .
	$$
	Dividing this inequality by $\log \alpha$, we get
	\begin{equation}\label{29}
		0<\left|\left(m_{1}+m_{2}+m_{3}\right) \frac{\log 10}{\log \alpha}-n-\frac{\log \left(9 /\left(2 d_{1}\right)\right)}{\log \alpha}\right|<(86.3) \cdot 10^{-m_{1}} .
	\end{equation}
	We can now invoke Lemma \ref{lem2}. Let us define
	$$
	\gamma:=\frac{\log 10}{\log \alpha} , \quad \mu:=-\frac{\log \left(9 /\left(2 d_{1}\right)\right)}{\log \alpha}, \quad A:=86.3, \quad B:=10, \quad \text { and } \quad w:=m_{1} .
	$$
	Let $M:=4 \cdot 10^{41}$. Then $M>m_{1}+m_{2}+m_{3}$, ensuring that the denominator of the $92$-th convergent of $\gamma$ is 13207611809473972496604686216431585910013086 exceeds $6 M$. We get the least positive
	$
	\epsilon:=\left\|\mu q_{92}\right\|-M\left\|\gamma q_{92}\right\|=0.062943
	$ for $d_1=6.$
	Consequently, the inequality \eqref{29} admits no solutions for
	$$
	m_{1} \geq 46.26>{\log \left(A q_{92} / \epsilon\right)}/{\log B}.
	$$
	Thus, we conclude that $m_{1} \leq 46$. Utilizing inequalities \eqref{24} and \eqref{28} together, and substituting this upper bound for $m_{1}$ into \eqref{28}, we obtain $n<4.4 \cdot 10^{29}$. Now, let us introduce
	$$
	z_{2}:=\log \left(\Gamma_4+1\right)=\left(m_{2}+m_{3}\right) \log 10-n \log \alpha-\log \left(\frac{9}{2\left(d_{1} 10^{m_{1}}-\left(d_{1}-d_{2}\right)\right)}\right).
	$$
	From \eqref{23}, we can assert that
	$
	\left|\Gamma_4\right|=\left|e^{-z_{2}}-1\right|<(1.11) \cdot 10^{-m_{2}}<0.2,
	$
	for $m_{2} \geq 1$. Choosing $a:=0.2$, we derive the inequality
	$$
	\left|z_{2}\right|=\left|\log \left(\Gamma_4+1\right)\right|<\frac{\log (5 / 4)}{0.2} \cdot \frac{1.11}{10^{m_{2}}}<\frac{1.24}{10^{m_{2}}}
	$$
	as confirmed by Lemma \ref{lem3}. This leads us to the conclusion that
	$$
	0<\left|\left(m_{2}+m_{3}\right) \log 10-n \log \alpha-\log \left(\frac{9}{2\left(d_{1} 10^{m_{1}}-\left(d_{1}-d_{2}\right)\right)}\right)\right|<1.24 \cdot 10^{-m_{2}} .
	$$
	Dividing both sides of the preceding inequality by $\log \alpha$, we get
	\begin{equation}\label{30}
		0<\left|\frac{\left(m_{2}+m_{3}\right) \log 10}{\log \alpha}-n-\frac{\log \left(9 /2\left(d_{1} 10^{m_{1}}-\left(d_{1}-d_{2}\right)\right)\right)}{\log \alpha}\right|<1.5 \cdot 10^{-m_{2}}. 
	\end{equation}
	Putting $\gamma:={\log 10}/{\log \alpha}$ and taking $m_{2}+m_{3}<M:=4.4\cdot 10^{29}$, we found that $q_{68}=27232938992914655197439992935676$, the denominator of the 68-th convergent of $\gamma$ exceeds $6M$. Next, we define
	$$
	\mu:=-\frac{\log \left(9 /(2\left(d_{1} 10^{m_{1}}-\left(d_{1}-d_{2}\right))\right)\right)}{\log \alpha} .
	$$
	Considering the constraints $m_{1} \leq 46, d_{1} \neq d_{2}, 1 \leq d_{1} \leq 9$ and $0 \leq d_{2} \leq 9$, a quick computation with $Mathematica$ yields the least positive
	$
	\epsilon=\epsilon(\mu):=\left\|\mu q_{68}\right\|-M\left\|\gamma q_{68}\right\|=0.0001994
	$
    for $(d_1, d_2, m_1)=(1,4,9).$
	Let $A:=1.5, B:=10$, and $w:=m_{2}$ in Lemma \ref{lem2}. Using $Mathematica$, we conclude that the inequality \eqref{30} has no solutions for
	$$
	m_{2} \geq 35.32>{\log \left(A q_{68} / \epsilon\right)}/{\log B}.
	$$
	Consequently, we have $m_{2} \leq 35$. Substituting the upper bounds obtained for $m_{1}$ and $m_{2}$ $(m_1\leq 46, m_2\leq35)$ into \eqref{28}, we get $n<1.3 \cdot 10^{16}$. Now, let
	$$
	z_{3}:=m_{3} \log 10-n \log \alpha+\log ({2\left(d_{1} 10^{m_{1}+m_{2}}-\left(d_{1}-d_{2}\right) 10^{m_{2}}-\left(d_{2}-d_{3}\right)\right)}/{9}).
	$$
	From \eqref{27}, we can express $\Gamma_5$ as
	$$
	\left|\Gamma_5\right|=\left|e^{z_{3}}-1\right|<1.02 \cdot \alpha^{-n}<0.01,
	$$
	valid for $n \geq 100$. Choosing $a:=0.01$, it follows that
	$$
	\left|z_{3}\right|=\left|\log \left(\Gamma_5+1\right)\right|<\frac{\log (100 / 99)}{0.01} \cdot \frac{1.02}{\alpha^{n}}<\frac{1.03}{\alpha^{n}}
	$$
	as per Lemma \ref{lem3}. This leads us to the conclusion:
	$$
	0<\left|m_{3} \log 10-n \log \alpha+\log ({2\left(d_{1} 10^{m_{1}+m_{2}}-\left(d_{1}-d_{2}\right) 10^{m_{2}}-\left(d_{2}-d_{3}\right)\right)}/{9})\right|<1.03 \cdot \alpha^{-n}.
	$$
	Dividing both sides of the above inequality by $\log \alpha$, we derive
	\begin{equation}\label{31}
		0<\left|m_{3} \frac{\log 10}{\log \alpha}-n+\frac{\log \left(2\left(d_{1} 10^{m_{1}+m_{2}}-\left(d_{1}-d_{2}\right) 10^{m_{2}}-\left(d_{2}-d_{3}\right)\right) / 9\right)}{\log \alpha}\right|<1.17 \cdot \alpha^{-n} .
	\end{equation}
	Let $\gamma:=\log 10/\log \alpha$ and consider $m_{3}<M:=1.3 \cdot 10^{16}$. We have found that $q_{43}=920197043232024959$, and notably, the denominator of the 43-th convergent of $\gamma$ exceeds $6M$. Defining
	$$
	\mu:=\frac{\log \left(2\left(d_{1} 10^{m_{1}+m_{2}}-\left(d_{1}-d_{2}\right) 10^{m_{2}}-\left(d_{2}-d_{3}\right)\right) / 9\right)}{\log \alpha}
	$$
	and taking into account the constraints $m_{1} \leq 46, m_{2} \leq 35,1 \leq d_{1} \leq 9$ and $0 \leq d_{2}, d_{3} \leq 9$, except in the scenarios where $d_{1}=d_{2} \neq d_{3},$ and $d_{1} \neq d_{2}=d_{3}$, a swift computation with $Mathematica$ reveals the least positive
	$$
	\epsilon=\epsilon(\mu):=\left\|\mu q_{43}\right\|-M\left\|\gamma q_{43}\right\|=0.00000123
	$$
	for $(d_1,d_2, d_3,m_1, m_2)=(3, 4, 0, 14, 5).$ Let us set $A:=1.17, B:=\alpha$, and $w:=n$ in Lemma \ref{lem2}. With the aid of $Mathematica$, we can confidently assert that the inequality \eqref{31} has no solution for
	$$
	n \geq 62.54>{\log \left(A q_{43} / \epsilon\right)}/{\log B}.
	$$
	Thus, we arrive at the conclusion $n \leq 62$, which stands in direct contradiction to our assumption that $n > 100$. This completes the proof.
\end{proof} 
\section{associated Pell numbers as the difference of two repdigits}
\begin{theorem}\label{thm5} The only associated Pell numbers that can be expressed as the difference of two repdigits are  1, 3, 7, 17, and 41, i.e.\\
	$q_0=q_1 = 1 = 9-8 = 8-7=7-6=6-5=5-4=4-3=3-2=2-1$, \\$q_2 = 3 = 11-8$,\\ $q_3 = 7 = 11-4$,\\$q_4 = 17 = 22-5$, \\and $q_5 = 41 = 44-3$.
	\begin{proof} Assume that \eqref{3} holds. Let $1 \leq n \leq 100$ and $n\geq 2$. Utilizing $Mathematica$, we find only the solutions detailed in Theorem \ref{thm5}. So from this point forward, we assume that $n > 100$. 
 If $k=l$, then it follows that $d_1 > d_2$, implying that $q_k$ is a repdigit. However, the largest possible repdigit in $q_n$ is 99 \cite{rp2018}. Thus, we get a contradiction since $n> 100$. Next, the scenario where $k-l=1$. If $d_1\geq d_2$, we encounter associated Pell numbers that are concatenation of two repdigits, which is impossible according to Lemma \ref{lem5}. If $d_1< d_2$, then we derive associated Pell numbers that are concatenation of three repdigits, contradicting Theorem \ref{thm4}. 
 
 Thus, we are left with the conclusion that $k-l \geq 2$. From the inequality
		\begin{equation*}
			\dfrac{\alpha^{2k}}{20}<\dfrac{10^{k-1}}{2}<10^{k-1}-10^{l-1}<\dfrac{d_1(10^k-1)}{9}-\dfrac{d_2(10^l-1)}{9}=q_n<\alpha^{n+1},
            \end{equation*}
       we obtain $\alpha^{2k}/20<\alpha^{n+1}$. Taking logarithms of both sides then yields the bound $k<n+5.$
        	
  We will now rewrite \eqref{3} into two distinct cases by using Binet's formula of associated Pell numbers, as presented below.
  \subsection{Case 1} Continuing with the first rearrangement of equation \eqref{3}, we obtain
		\begin{equation}\label{32}
			\dfrac{9\alpha^n}{2}- d_1 10^k =  -\biggr(\dfrac{9\beta^n}{2}+ d_2 10^l +(d_1-d_2)\biggr).
		\end{equation}
		Taking the absolute value of both sides of \eqref{32}, we obtain
		\begin{equation}\label{33}
			\biggr|\dfrac{9\alpha^n}{2}-d_1 10^k\biggr| \leq \dfrac{9|\beta|^n}{2}+d_2 10^l + |d_1-d_2|.
		\end{equation}
		Dividing both sides of \eqref{33} by $d_1 10^k$, we obtain
		\begin{equation*}
			\begin{split}
				\biggr|\dfrac{9\cdot10^{-k}\cdot\alpha^n}{2d_1}-1\biggr| & \leq \dfrac{9|\beta|^n}{2d_1 10^k} + \dfrac{d_2 10^l}{d_1 10^k}+\dfrac{|d_1-d_2|}{d_1 10^k} \\ 
				& \leq \dfrac{9|\beta|^n}{2\cdot10^{k-l+1}} + \dfrac{9}{10^{k-l}}+\dfrac{8}{10^{k-l+1}} .
			\end{split}
		\end{equation*}
		From this, we conclude that
		\begin{equation} \label{34}
			\biggr|\dfrac{9\cdot10^{-k}\cdot\alpha^n}{2d_1}-1\biggr| < \dfrac{9.81}{10^{k-l}}.
		\end{equation}
		Next, we apply Theorem \ref{thm1} with $(\gamma_1, \gamma_2, \gamma_3)$ = $(\alpha, 10, 9/2d_1)$ and $(b_1, b_2, b_3)$ = $(n, -k, 1)$. Notably, $\gamma_1$, $\gamma_2$, and $\gamma_3$ are positive real numbers and elements of the field $\mathbb K = \mathbb Q(\sqrt{2})$. Consequently, the degree of the field
		$\mathbb K$ is equal to $d_{\mathbb L}= 2$. Let
		\begin{equation*}
			\Gamma_6 = \dfrac{9\cdot10^{-k}\cdot\alpha^n}{2d_1}-1 .
		\end{equation*}
		We can ensure that $\Gamma_6 \neq 0$ as per earlier arguments.
		Leveraging the properties of absolute logarithmic height, we can analyze the heights as \begin{center}
		    $h(\gamma_1)=\dfrac{\log\alpha}{2},$ $h(\gamma_2)= \log 10,$ and $h(\gamma_3)\leq h(2d_1)+h(9) < 5.09 . $\end{center}
		We can set $A_1=\log \alpha$, $A_2= 2\log 10$, $A_3= 10.18$. Since $k< n+5$ and $D$ $\geq$ max$\{n, k, 1\}$, we can conveniently choose $D= n+5$. 
		Considering equation \eqref{34} and implementing Theorem \ref{thm1}, we obtain
		\begin{equation*}
			\log\biggr(\dfrac{9.81}{10^{k-l}}\biggr)>\log|\Gamma_6| > -1.4\cdot30^6\cdot3^{4.5}\cdot2^2 (1+\log 2)(1+\log (n+5))(\log\alpha)(2\log 10)(10.18).
		\end{equation*}
		A straightforward calculation reveals that this inequality leads to 
		\begin{equation}\label{35}
			\begin{split}
				(k-l)\log10 &< \log (9.81) + 4.1 \cdot 10^{13} (1+\log (n+5))\\ & < 4.2 \cdot 10^{13} (1+\log (n+5)).
			\end{split}
		\end{equation}
  \subsection{Case 2} Advancing to the second rearrangement of \eqref{3} as
		\begin{equation} \label{36}
			\dfrac{\alpha^n}{2}- \dfrac{d_1 10^k - d_2 10^l}{9} = -\dfrac{\beta^n}{2}- \dfrac{(d_1 - d_2)}{9}
		\end{equation}
		and taking the absolute value of both sides of \eqref{36}, we arrive at
		\begin{equation}\label{37}
			\biggr|\dfrac{\alpha^n}{2}- \dfrac{d_1 10^k - d_2 10^l}{9}\biggr| \leq \dfrac{|\beta^n|}{2}+ \dfrac{|d_1 - d_2|}{9}.
		\end{equation}
		Dividing both sides of the above inequality by $\alpha^n/2$, we find
		\begin{equation}\label{38}
			\biggr|1- \dfrac{2(d_1-d_210^{l-k})\cdot10^k\cdot\alpha^{-n}}{9}\biggr| \leq \frac{1}{\alpha^{2n}}+\dfrac{16}{9\alpha^n}< \frac{3}{\alpha^n}.
		\end{equation}
		We can now apply Theorem \ref{thm1} to the above inequality with
		\begin{center}
			$(\gamma_1, \gamma_2, \gamma_3)$ = $\biggr(\alpha, 10, \dfrac{2(d_1-d_210^{l-k})}{9}\biggr)$ and $(b_1, b_2, b_3)$ = $(-n, k, 1)$.
		\end{center}
		Crucially, $\gamma_1$, $\gamma_2$, and $\gamma_3$ are positive real numbers that lie within the field $\mathbb K = \mathbb Q(\sqrt{2})$. Thus, the degree of the field $\mathbb K$ is $d_{\mathbb L}= 2$.
		Let
		\begin{equation*}
			\Gamma_7 = 1- \dfrac{2(d_1-d_210^{l-k})\cdot10^k\cdot\alpha^{-n}}{9}.
		\end{equation*}
		If $\Gamma_7=0$, then $$1= \dfrac{2(d_1-d_210^{l-k})\cdot10^k\cdot\alpha^{-n}}{9}.$$ This leads to $(d_1-d_210^{l-k})\cdot10^k\cdot\alpha^{-n}=9/2$, implying $\alpha^{n}\in \mathbb Q$, which is a contradiction for $n>0$. 
  Using properties of absolute logarithmic height, we obtain
		\begin{equation*}
			h(\gamma_1)=h(\alpha)={(\log\alpha)}/{2},\hspace{0.2cm}
			h(\gamma_2)= \log 10. 
		\end{equation*}
		Next, we will estimate $h(\gamma_3)= h({2(d_1-d_210^{l-k})}/{9})$.
  Applying the properties of absolute logarithmic heights, we obtain
		\begin{equation*}
			\begin{split}
				h(\gamma_3) &\leq h(d_1/9)+h(d_2/9)+(k-l)\log10 +\log 2+\log 2 \\ & \leq 6.48+(k-l) \log 10.
			\end{split}
		\end{equation*}
		With these heights established, we can define $$A_1=\log \alpha, A_2= 2\log 10, A_3= 12.96+2(k-l)\log 10.$$ Since $k< n+5$ and $D$ $\geq \max\{n, k, 1\}$, we can take $D= n+5$. 
		Thus, considering \eqref{38} and applying Theorem \ref{thm1}, we derive
		\begin{equation*}
			3\cdot\alpha^n>|\Gamma_7|> e^{(C\cdot(1+\log 2)(1+\log(n+5))\cdot\log \alpha\cdot2\log 10\cdot(15.96+2(k-l)\log 10))},
		\end{equation*}
		where $C= -1.4\cdot30^6\cdot3^{4.5}\cdot 2^2$. By a simple computation, it follows that 
		\begin{equation}\label{39}
			n\log\alpha-\log 3 <4\cdot10^{12}\cdot(1+\log(n+5))(12.96+2(k-l)\log 10).
		\end{equation}
		Utilizing \eqref{35} and \eqref{39}, a computational search with $Mathematica$ gives us $n< 1.4\cdot10^{28}$.
	\subsection{Reducing the upper bound on $k$}	
		Let us reduce the upper bound on $k$ by using the Baker–Davenport algorithm as described in Lemma \ref{lem1}. We define
		\begin{equation*}
			z_3=n\log\alpha - k\log10 +\log({9}/{2d_1}).
		\end{equation*}
		\vspace{0.1cm}From \eqref{34}, we have
		\begin{equation*}
			|x| = |e^{z_3}-1|< \dfrac{9.81}{10^{k-l}}<\frac{1}{10}
		\end{equation*}
		for $k-l \geq 2$. Choosing $a = 0.1$, we arrive at the inequality
		\begin{equation*}
			|z_3|=|\log(x+1)|< \dfrac{\log(10/9)}{1/10} \cdot \frac{9.81}{10^{k-l}}<10.34\cdot 10^{l-k}
		\end{equation*}
		by Lemma \ref{lem2}. Consequently, we deduce that
		\begin{equation*}
			0<\biggr|n\log\alpha - k\log10 +\log({9}/{2d_1})\biggr|< 10.34\cdot 10^{l-k}.
		\end{equation*}
		Dividing this inequality by $\log 10$, we obtain
		\begin{equation}\label{40}
			0<\biggr|n\biggr(\frac{\log\alpha}{\log 10}\biggr) - k +\dfrac{\log({9}/{2d_1})}{\log 10}\biggr|< 4.5 \cdot 10^{l-k}.
		\end{equation}
		We can select $\tau={\log\alpha}/{\log 10} \notin \mathbb Q$ and $M = 1.4\cdot 10^{28}$. Notably, we find that $q_{66} = 86117281818724112510090871404
$, the denominator of the 66-th convergent of $\tau$ exceeding $6M$.
		Next, we define $\mu=\dfrac{\log(9/2d_1)}{\log 10}$.
		In this case, considering the fact that $1 \leq d_1 \leq9$, a quick computation with $Mathematica$ reveals the least positive  $$\epsilon(\mu) := \left\lVert \mu q_{66}\right\rVert - M\left\lVert \tau q_{66}\right\rVert = 0.138816$$
		for $d_1=8$. Let $A=4.5$, $B= 10$, and $\omega=k-l$ in Lemma \ref{lem1}. Thus, employing $Mathematica$, we can say that \eqref{40} has no solution if 
		\begin{equation*}
			\dfrac{\log(Aq_{66}/\epsilon(\mu))}{\log B} < 30.44<k-l.
		\end{equation*}
		So $k-l \leq 30 $.
		Substituting this upper bound for $k-l$ in \eqref{39}, we derive the result $n < 3.9 \cdot 10^{15} $. Now, let
		\begin{equation*}
			z_4= k\log10-n\log\alpha +\log\biggr(\dfrac{2(d_1-d_210^{l-k})}{9}\biggr).
		\end{equation*}
		From \eqref{38}, we have 
		\begin{equation*}
			|x| = |e^{z_4}-1|< \dfrac{3}{\alpha^n}<\frac{1}{10}
		\end{equation*}
		for $k\geq 25$. Choosing $a = 0.1$, we get the inequality
		\begin{equation*}
			|\Lambda_2|=|\log(x+1)|< \dfrac{\log(10/9)}{1/10} \cdot \frac{3}{\alpha^n}<3.17\cdot \alpha^{-n}
		\end{equation*}
		by Lemma \ref{lem2}. Thus, we conclude
		\begin{equation*}
			0<\biggr|k\log10-n\log\alpha +\log\biggr(\dfrac{2(d_1-d_210^{l-k})}{9}\biggr)|< 3.17\cdot \alpha^{-n}.
		\end{equation*}
		Dividing both sides by $\log \alpha$, we obtain
		\begin{equation}\label{41}
			0<\biggr|k\biggr(\frac{\log 10}{\log \alpha}\biggr) - n +\dfrac{\log\biggr(\dfrac{2(d_1-d_210^{l-k})}{9}\biggr)}{\log \alpha}\biggr|< 3.6\cdot \alpha^{-n}.
		\end{equation}
		Putting $\tau =\log 10/\log \alpha$ and taking $M= 3.9\cdot10^{15}$, we found that $q_{41}=30910886367884945$, the denominator of the 40-th convergent of $\tau$ exceeds $6M$. Now set
		$$\mu = \dfrac{\log\biggr(\dfrac{2(d_1-d_210^{l-k})}{9}\biggr)}{\log \alpha}.$$
		In this scenario, noting that $1 \leq d_1, d_2 \leq 9$ and $2 \leq k-l \leq 30$, a quick computation yields the least positive $\epsilon := \left\lVert \mu q_{41}\right\rVert - M\left\lVert \tau q_{41}\right\rVert =0.0.000564$ for $(d_1, d_2, k-l)=(4,1,11).$ Let $A=3.6$, $B= \alpha$, and $\omega=n$ in Lemma \ref{lem1}. Therefore, utilizing $Mathematica$, we find that equation \eqref{41} has no solution if 
		\begin{equation*}
			\dfrac{\log(Aq_{40}/\epsilon)}{\log B} < 53.019< n.
		\end{equation*}
		This leads to the result $n \leq 53$, which contradicts our assumption that $n> 100$. Hence, the proof is complete.
	\end{proof}	
\end{theorem}
{\bf Data Availability Statements:} Data sharing is not applicable to this article as no datasets were generated or analyzed during the current study.

{\bf Funding:} The authors declare that no funds or grants were received during the preparation of this manuscript.

{\bf Declarations:}

{\bf Conflict of interest:} On behalf of all authors, the corresponding author states that there is no {conflict} of interest.

\end{document}